\documentclass{siamart220329}
\usepackage{mathrsfs}

\usepackage{braket,amsfonts}

\usepackage{array}

\usepackage[caption=false]{subfig}
\usepackage{pgfplots}

\newsiamthm{claim}{Claim}
\newsiamremark{remark}{Remark}
\newsiamremark{hypothesis}{Hypothesis}
\crefname{hypothesis}{Hypothesis}{Hypotheses}

\usepackage{algorithmic}

\usepackage{graphicx,epstopdf}

\Crefname{ALC@unique}{Line}{Lines}

\usepackage{amsopn}

\usepackage{xspace}
\usepackage{bold-extra}
\usepackage[most]{tcolorbox}

\colorlet{texcscolor}{blue!50!black}
\colorlet{texemcolor}{red!70!black}
\colorlet{texpreamble}{red!70!black}
\colorlet{codebackground}{black!25!white!25}

\lstdefinestyle{siamlatex}{%
  style=tcblatex,
  texcsstyle=*\color{texcscolor},
  texcsstyle=[2]\color{texemcolor},
  keywordstyle=[2]\color{texemcolor},
  moretexcs={cref,Cref,maketitle,mathcal,text,headers,email,url},
}

\tcbset{%
  colframe=black!75!white!75,
  coltitle=white,
  colback=codebackground, 
  colbacklower=white, 
  fonttitle=\bfseries,
  arc=0pt,outer arc=0pt,
  top=1pt,bottom=1pt,left=1mm,right=1mm,middle=1mm,boxsep=1mm,
  leftrule=0.3mm,rightrule=0.3mm,toprule=0.3mm,bottomrule=0.3mm,
  listing options={style=siamlatex}
}

\newtcblisting[use counter=example]{example}[2][]{%
  title={Example~\thetcbcounter: #2},#1}

\newtcbinputlisting[use counter=example]{\examplefile}[3][]{%
  title={Example~\thetcbcounter: #2},listing file={#3},#1}

\DeclareTotalTCBox{\code}{ v O{} }
{ 
  fontupper=\ttfamily\color{black},
  nobeforeafter,
  tcbox raise base,
  colback=codebackground,colframe=white,
  top=0pt,bottom=0pt,left=0mm,right=0mm,
  leftrule=0pt,rightrule=0pt,toprule=0mm,bottomrule=0mm,
  boxsep=0.5mm,
  #2}{#1}

\patchcmd\newpage{\vfil}{}{}{}
\def\cB{\mathcal{B}}
\def\sE{\mathscr{E}}
\def\sF{\mathscr{F}}

\def\cJ{\mathcal{J}}
\def\sK{\mathscr{K}}
\def\fm{\mathfrak{m}}
\def\cM{\mathcal{M}}
\def\NN{\mathbb{N}}
\def\fp{\mathfrak{p}}

\def\fq{\mathfrak{q}}
\def\sQ{\mathscr{Q}}

\def\RR{\mathbb{R}}
\def\cR{\mathcal{R}}
\def\cU{\mathcal{U}}
\def\ZZ{\mathbb{Z}}

\DeclareMathOperator{\reg}{reg}
\DeclareMathOperator{\rank}{rank}

\title{A Field-Theoretic View of Unlabeled Sensing}

\author{Hao Liang \, \,  Jingyu Lu
\, \,  Manolis C. Tsakiris \, \,  Lihong Zhi \, \thanks{All authors are affiliated with the Key Laboratory of Mathematics Mechanization, Academy of Mathematics and Systems Science, Chinese Academy of Sciences, Beijing, 100190, China. Emails: \email{lianghao2020@amss.ac.cn}, \email{lujingyu@amss.ac.cn}, \email{manolis@amss.ac.cn}, \email{lzhi@mmrc.iss.ac.cn}.}}

\begin{document}

\maketitle
\begin{abstract}
Unlabeled sensing is the problem of solving a linear system of equations, where the right-hand-side vector is known only up to a permutation. In this work, we study fields of rational functions related to symmetric polynomials and their images under a linear projection of the variables; as a consequence, we establish that the solution to an $n$-dimensional unlabeled sensing problem with generic data can be obtained as the unique solution to a system of $n+1$ polynomial equations of degrees $1, \, 2, \, \dots, n+1$ in $n$ unknowns. Besides the new theoretical insights, this development offers the potential for scaling up algebraic unlabeled sensing algorithms.
\end{abstract}

\section{Introduction} \label{section:Introduction}

\subsection{Unlabeled Sensing}
	
In \emph{unlabeled sensing} \cite{Unnikrishnan-Allerton2015, unnikrishnan2018unlabeled} one is given a tall full-rank matrix $A^* \in \RR^{m \times n}$ and a vector $y^* \in \RR^m$, such that for a permutation $\pi$ of the coordinates of $\RR^m$ the linear system of equations $A^* x = \pi(y^*)$ has a solution $\xi^*$; the problem then is to find $\xi^*$ from $A^*$ and $y^*$. The main theorem of unlabeled sensing asserts that this is a well-defined question when $A$ is generic and $m \ge 2n$. This result has been generalized beyond permutations to arbitrary linear transformations \cite{TSAKIRIS2023210,Peng-ACHA-21}, as well as beyond linear spaces to unions of linear spaces \cite{Peng-ACHA-21} and to spaces of bounded-rank matrices \cite{yao2021, UPCA-JMLR, Tsakiris-ISIT23}.

Unlabeled sensing is an extremely challenging computational problem, known to be NP-hard \cite{Pananjady-TIT18,Hsu-NIPS17}, with brute-force \cite{Elhami-ICASSP17} or globally optimal approaches being tractable only for small dimensions; see \cite{peng_linear_2020} for a brief account. Nevertheless, unlabeled sensing has a wealth of potential applications from biology \cite{Abid-Allerton2018,ma2021bacterial} and neuroscience \cite{nejatbakhsh2021neuron} to digital communications \cite{Song-ISIT18}, data mining \cite{Slawski-JoS19,Slawski-JMLR2020,zhang2021benefits} and computer vision \cite{Tsakiris-ICML2019,li2023shuffled}.

\subsection{Motivation} 
In this paper we are concerned with algebraic aspects of unlabeled sensing \cite{Song-ISIT18,Tsakiris-TIT2020,melanova2022recovery}, for which we now set the context. Let 
\begin{equation*}
\RR[y_1,\dots,y_m]=:\RR[y], \, \, \, \RR[x_1,\dots,x_n]=:\RR[x]
\end{equation*}
be polynomial rings in $m$ and $n$ variables respectively over the real numbers $\RR$, and let
\begin{equation*}
    p_\ell = \sum_{i \in [m]} y_i^\ell \in \RR[y]
\end{equation*}
be the $\ell$-th power sum of the $y_i$'s; here and in the sequel $[t]$ denotes the set $\{1,2,\dots,t\}$, whenever $t$ is a positive integer. In \cite{Song-ISIT18} it was observed that $\xi^*$ is a root of the polynomial
\begin{equation*}
    q_\ell = p_\ell(A^* x)-p_\ell(y^*) \in \RR[x]
\end{equation*}
for any $\ell \in \NN$; here and in the sequel $x$ is the column vector containing $x_1,\dots,x_n$ in its entries. With $A^*$ generic, it was proved in \cite{Tsakiris-TIT2020} that the square system 
\begin{equation*}
    \sQ_n: \, q_1(x)=\cdots=q_n(x)=0
\end{equation*}
is zero-dimensional and thus has at most $n!$ solutions. An algorithm was also developed, which involved solving the square polynomial system for all of its roots via off-the-shelf solvers, isolating a root by a suitable criterion, and then using an expectation-maximization procedure to refine that root. An attractive feature of this algorithm is that it has linear complexity in $m$, while it has been empirically observed to be robust to low levels of noise: for SNR=$40$dB, $m=1000$ and $n=4$, the algorithm took $25$ milliseconds on a standard PC to produce a solution with a relative error of $0.4\%$ with respect to the ground truth. On the other hand, this algorithm is not scalable with respect to $n$: as $n$ increases, one would not even be able to store efficiently the $n!$ solutions of the square polynomial system, let alone solve it; indeed, in \cite{Tsakiris-TIT2020} it was possible to report results only for $n \le 6$.

\subsection{Contributions}
Our object of study in this paper is the overdetermined system of $n+1$ polynomial equations in $n$ unknowns 
\begin{equation*}
    \sQ_{n+1}: \, q_1(x) = \cdots = q_{n+1}(x)=0.
\end{equation*}
Our main result reads:

\begin{theorem} \label{thm:main}
Suppose that $A^* \in \RR^{m \times n}$ and $\xi^* \in \RR^{n \times 1}$ are generic, let $\pi$ be any permutation of the coordinates of $\RR^{m \times 1}$,
and set $y^* = \pi(A^* \xi^*) \in \RR^{m \times 1}$. Then $\xi^*$ is the unique complex solution of the polynomial system $\sQ_{n+1}$.
\end{theorem}

Theorem \ref{thm:main} settles an important open question in the theory of unlabeled sensing. Indeed, that $\sQ_{n+1}$ has a unique solution for generic data was already experimentally observed in \cite{Song-ISIT18}, and a more general unique recovery conjecture was formulated in \cite{melanova2022recovery} (Conjecture 6). But Theorem \ref{thm:main} also has significant implications for unlabeled sensing algorithms: no matter which method is used to obtain a root of $\sQ_{n+1}$, Theorem \ref{thm:main} guarantees that this root is $\xi^*$; contrast this to the algorithm of \cite{Tsakiris-TIT2020} which relied on filtering all $n!$ solutions of $\sQ_n$. Indeed, in \cite{liang2024unlabeled}, for which the present manuscript partially serves as a rigorous theoretical foundation, we have proposed an algorithm for obtaining the unique solution of $\sQ_{n+1}$ via rank-$1$ moment matrix completion, and reported encouraging results.

The proof of Theorem \ref{thm:main} relies on a careful analysis of a certain field of rational functions associated to the polynomials $p_\ell$ after a linear projection of the variables $x$ has been applied, which is interesting on its own right (note that without the linear projection, the study of the field generated by the $p_\ell$'s for fixed given values of $\ell$ and the question of when this coincides with the field of symmetric rational functions on $x$, is an old and well-known topic in the literature, e.g. see \cite{Kakeya1927, Nakamura1927, Foulkes1956,Dvornicich&Zannier2009}). Let $A=(a_{ij})$ be an $m \times n$ matrix of variables so that all $a_{ij}$'s and $x_j$'s are jointly algebraically independent over $\RR$. Denote by $\sE = \RR(A,x)$ the field of rational functions in the variables $A,x$ with coefficients in $\RR$ and $\mathscr{F}_{n+1} = \RR(A, p_1(Ax),\dots,p_{n+1}(Ax))$ the subfield of $\sE$ consisting of the rational functions in $A, p_1(Ax),\dots,p_{n+1}(Ax)$ with coefficients in $\RR$. We prove:

\begin{theorem} \label{thm:fields}
The fields of rational functions $\sF_{n+1}$ and $\sE$ coincide.
\end{theorem}

\subsection{Acknowledgments} We thank Aldo Conca for reading the manuscript and giving us valuable comments and Matteo Varbaro for a pointer in the literature. {\color{black} Manolis C. Tsakiris is supported by the National Key R$\&$D Program of China (2023YFA1009402).  Lihong  Zhi is  supported by the National Key R$\&$D Program of China (2023YFA1009401) and the National Natural
Science Foundation of China (12071467). }

\section{The Field $\sF_m$}

Nothing of what we will say in this and the next section depends on the ground field, other than the requirement that it has characteristic zero; we thus fix throughout such a field $k$; we will denote by $k(A)$ the field of fractions of the polynomial ring $k[A]$. We fix polynomial rings over the field $k(A)$
\begin{equation*}
    S := k(A)[z_1,\dots,z_m], \, T := k(A)[y_1,\dots,y_m], \, R:=k(A)[x_1,\dots,x_n],
\end{equation*}
and define $k(A)$-algebra homomorphisms $ S \stackrel{\varphi}{\rightarrow} T \stackrel{\psi}{\rightarrow} R $, where $\varphi(z_i) = p_i$ and $\psi(y_i) = \sum_{j \in [n]} a_{ij} x_j$ for every $i \in [m]$. We let $\fp = I_{n+1}(A | y)$ be the ideal of $T$ generated by all $(n+1)$-minors of the $m \times (n+1)$ matrix $[A | y]$; here $y$ is the vector of variables $y_1,\dots,y_m$. Each such $(n+1)$-minor is a linear form of $T$ and thus $\fp$ is a prime ideal of $T$ whose height can be seen to be $m-n$. Since $\psi$ is surjective,  $\fp = \ker (\psi)$ and $R \cong T / \fp$.

Let $\fq = S \cap \fp$ be the contraction of $\fp$ to $S$ under $\varphi$; this is a prime ideal of $S$ whose residue field $S_\fq / \fq S_\fq$ we denote by $\kappa(\fq) $. We have an inclusion of integral domains $S/\fq \hookrightarrow T / \fp \cong R = k(A)[x]$, which identifies $S/\fq$ with the $k(A)$-subalgebra $k(A)[p_1(Ax),\cdots,p_m(Ax)]$ of $R$; we denote the field of fractions of this subalgebra by $\sF_m$. In turn, this induces an inclusion of fraction fields $\kappa(\fq) \cong \sF_m \subseteq \sE \cong \kappa(\fp)$, where we recall that $\sE$ is the field of fractions of the polynomial ring $k(A)[x]$, and $\kappa(\fp) = T_\fp / \fp T_\fp$. The main result of this section is:

\begin{proposition} \label{prp:Fm=E}
We have an equality of fields $\sF_m  = \sE$.
\end{proposition}

Towards proving Proposition \ref{prp:Fm=E}, we begin with a basic but important fact.

\begin{lemma} \label{lem:m!}
$T$ is a free $S$-module of rank $m!$.
\end{lemma}
\begin{proof}
By virtue of Newton's identities, the subalgebra $\varphi(S)=k(A)[p_1,\dots,p_m]$ of $T$ coincides with the subalgebra $k(A)[s_1,\dots,s_m]$ generated by the $m$ elementary symmetric functions $s_1,\dots,s_m$ on the variables $y$; it thus suffices to prove that $T$ is a free module over $k(A)[s_1,\dots,s_m]$ of rank $m!$.

Let us consider the polynomial ring 
\begin{equation*}
    P=k(A)[y,w] = k(A)[y_1,\dots,y_m,w_1,\dots,w_m]
\end{equation*} 
of dimension $2m$, and the ideal $J$ of $P$ generated by all $w_i - s_i(y)$ for $i \in [m]$, where $s_i(y)$ is the $i$-th elementary symmetric function on the variables $y$. Under any monomial order on $P$ with $y_1>\cdots>y_m>w_1>\cdots>w_m$, Proposition 5 in Chapter 7 of \cite{Cox-2013} explicitly describes a Gr\"obner basis of $J$, consisting of $m$ polynomials $g_1,\dots,g_m \in P$ such that the leading term of $g_i$ is $y_i^i$. It immediately follows that a $k(A)$-vector space basis of $P/J$ is given by all monomials of the form $y_2^{\ell_2} \cdots y_m^{\ell_m} w_1^{b_1}\cdots w_m^{b_m}$, where the $b_i$'s range over the non-negative integers while $0 \le \ell_i < i$.

Now consider the $k(A)$-algebra epimorphism $\theta: P=k(A)[y,w] \rightarrow T = k(A)[y]$ defined by $\theta(w_i) = s_i(y)$ and $\theta(y_i)=y_i$ for every $i \in [m]$. We have that $J = \ker(\theta)$ and so $T = P/J$; since $P/J$ is generated over $k(A)$ by the $m!$ monomials $y_2^{\ell_2} \cdots y_m^{\ell_m}$ as above and all monomials in $w$, $T$ is a fortiori generated over $k(A)[s_1(y),\dots,s_m(y)]\cong k(A)[w_1,\dots,w_m]$ by the $y_2^{\ell_2} \cdots y_m^{\ell_m}$'s. To show these monomials are free generators, suppose there is a relation $\sum_{0 \le \ell_i < i} c_{\ell_2,\dots, \ell_m}(y) y_2^{\ell_2} \cdots y_m^{\ell_m} =0$ in $T$, with $c_{\ell_2,\dots, \ell_m}(y) \in k(A)[s_1(y),\dots,s_m(y)]$. Write $c_{\ell_2,\dots, \ell_m}(y) = f_{\ell_2,\dots, \ell_m}(s_1(y),\dots,s_m(y))$, where $f_{\ell_2,\dots, \ell_m}$ is a polynomial in $m$ variables with coefficients in $k(A)$. Note $\theta(f_{\ell_2,\dots, \ell_m}(w)) = c_{\ell_2,\dots, \ell_m}(y)$, whence $\sum_{0 \le \ell_i < i} f_{\ell_2,\dots, \ell_m}(w) y_2^{\ell_2} \cdots y_m^{\ell_m} \in \ker(\theta)$. Since $J = \ker(\theta)$ and the $g_i$'s form a Gr\"obner basis of $J$, the leading term of $\sum_{0 \le \ell_i < i} f_{\ell_2,\dots, \ell_m}(w) y_2^{\ell_2} \cdots y_m^{\ell_m}$ must be divisible by $y_i^i$ for some $i \in [m]$; however, it is seen from the form of $\sum_{0 \le \ell_i < i} f_{\ell_2,\dots, \ell_m}(w) y_2^{\ell_2} \cdots y_m^{\ell_m}$ that this is impossible, unless this is the zero polynomial. Since all monomials in $y$ freely generate $k(A)[y,w]$ as a module over $k(A)[w]$, we in turn have that all $f_{\ell_2,\dots, \ell_m}(w)$'s and thus all $c_{\ell_2,\dots, \ell_m}(y)$'s are zero.
\end{proof}

\begin{remark}
The fact that $T$ is a free $S$-module is a special case of the well-known Chevalley-Shephard-Todd theorem.
\end{remark}

Lemma \ref{lem:m!} implies that $T$ is integral over $\varphi(S)$; this is a manifestation of a general fact:

\begin{lemma}[Exercises 12 \& 13, Chapter 5, \cite{Atiyah-MacDonald}] \label{lem:Atiyah-MacDonald}
Let $R$ be a commutative ring and $\Pi$ a finite group acting on $R$; denote by $R^\Pi$ the subring of $R$ consisting of the invariant elements of $R$ with respect to the action of $\Pi$. Then the ring extension $R^\Pi \subseteq R$ is integral, and for any prime ideal $\fp$ of $R$, the prime ideals of $R$ that lie over $\fp \cap R^\Pi$ are the orbit $\{\pi(\fp): \, \pi \in \Pi\}$ of $\fp$ under $\Pi$.
\end{lemma}

We have:

\begin{lemma} \label{lem:lying over q}
The prime ideals of $T$ that lie over $\fq = S \cap \fp$ are precisely of the form $\pi(\fp)$, where $\pi$ is a permutation of the variables $y_1,\dots,y_m$; these are $m!$ distinct prime ideals.
\end{lemma}
\begin{proof}
That a prime ideal of $T$ lies over $\fq$ if and only if it is of the form $\pi(\fp)$, follows from Lemma \ref{lem:Atiyah-MacDonald}. We prove that all $m!$ such prime ideals $\pi(\fp)$ are distinct. For this, it is enough to prove that $\pi(\fp) \neq \fp$ as soon as $\pi$ is not the identity permutation; without loss of generality, we may assume that $\sigma (\cJ) \neq \cJ$ for $\cJ = \{m-n+1,\dots,m\}$ and $\sigma = \pi^{-1}$.

For $i_1,\dots,i_n,i_{n+1}$ distinct elements of $m$, we denote by $A_{i_1,\dots,i_n}$ the determinant of the $n \times n$ matrix, whose $s$-th row is the $i_s$-th row of $A$, and by $A_{i_1,\dots,\hat{i}_s,\dots,i_{n+1}}$ the determinant as above associated to rows $i_1,\dots,i_{s-1},i_{s+1},\dots,i_{n+1}$. With this notation, for any monomial order on $T$ with $y_1 > y_2 > \cdots > y_m$, the $m-n$ linear forms of $T$ given by
\begin{equation*}
    \ell_i := y_i + \sum_{s \in \cJ} (-1)^{s-m+n+1}  \frac{\det(A_{i,m-n+1,...,\hat{s},\dots,m})}{\det(A_{m-n+1,\dots,m})} y_{s}, \, \, \, i \in [m-n]
\end{equation*}
\noindent are a reduced Gr\"obner basis of $\fp$. Similarly for $i \in [m-n]$, the linear forms
\begin{equation*}
    \mu_i := y_i + \sum_{s \in \sigma(\cJ)} (-1)^{s-m+n+1}  \frac{\det(A_{\sigma(i),\sigma(m-n+1),...,\hat{s},\dots,\sigma(m)})}{\det(A_{\sigma(m-n+1),\dots,\sigma(m)})} y_{s}
\end{equation*}
\noindent are a reduced Gr\"obner basis of $\pi(\fp)$. Since a reduced Gr\"obner basis is unique, if $\fp = \pi(\fp)$, necessarily $\ell_i = \mu_i$ for every $i \in [m-n]$. In particular, for any $i \in [m-n]$ and any $s \in \cJ$, the coefficient of $y_s$ in $\ell_i$, must be equal to the coefficient of $y_s$ in $\mu_i$.

By our hypothesis $\sigma (\cJ) \neq \cJ$, there exist $s \in \cJ \setminus \sigma (\cJ)$ and $i \in \sigma (\cJ) \setminus \cJ \subseteq [m-n]$. The coefficient of $y_s$ in $\ell_i$ is $(-1)^{s-m+n+1}  \det(A_{i,m-n+1,...,\hat{s},\dots,m})/\det(A_{m-n+1,\dots,m})$. On the other hand, since $i \in [m-n]$ and $s \neq i$ with $s \not\in \sigma (\cJ)$, the coefficient of $y_s$ in $\mu_i$ is zero.
\end{proof}

Proposition \ref{prp:Fm=E} is a special case of the following result for $\mathfrak{P}=\fp$:

\begin{lemma}
Let $\mathfrak{P}$ be a prime ideal of $T$ lying over $\fq$. Then $\kappa(\mathfrak{P}) = \kappa(\fq)$.
\end{lemma}
\begin{proof}
As $T$ is a free $S$-module of rank $m!$ by Lemma \ref{lem:m!}, $T \otimes_S \kappa(\fq)$ is a $\kappa(\fq)$-vector space of dimension $m!$. Since $T \otimes_S \kappa(\fq)$ is a finitely generated $\kappa(q)$-algebra, which is also a finite-dimensional $\kappa(\fq)$-vector space, it must be an Artinian ring. The prime ideals of $T \otimes_S \kappa(\fq)$ correspond to the prime ideals of $T$ lying over $\fq $. By Lemma \ref{lem:lying over q}, these are the $m!$ distinct prime ideals $\pi(\fp)$, with $\pi$ ranging over all permutations of the variables $y_1,\dots,y_m$. Quite generally, an Artinian ring is isomorphic to the product of its localizations at its prime ideals, hence
\begin{equation*}
    T \otimes_S \kappa(\fq) = \prod_{\pi} T_{\pi(\fp)} \otimes_S \kappa(\fq).
\end{equation*}
\noindent Now, each $T_{\pi(\fp)} \otimes_S \kappa(\fq) = T_{\pi(\fp)} / \fq T_{\pi(\fp)}$ is an Artinian local ring and a finite $\kappa(\fq)$-vector space. Since $T_{\pi(\fp)} \otimes_S \kappa(\fq)$ is an $m!$-dimensional $\kappa(\fq)$-vector space and there are $m!$ factors in the product, it must be that each $T_{\pi(\fp)} \otimes_S \kappa(\fq)$ is a $1$-dimensional $\kappa(\fq)$-vector space, for every $\pi$. But $\kappa(\fq)$ is contained in every $T_{\pi(\fp)} \otimes_S \kappa(\fq)$, so that $T_{\pi(\fp)} \otimes_S \kappa(\fq) = \kappa(\fq)$ for every $\pi$. Since $T_{\pi(\fp)} / \fq T_{\pi(\fp)}$ is a field, it must be that $\fq T_{\pi(\fp)} = \pi(\fp) T_{\pi(\fp)}$ and so $T_{\pi(\fp)} / \fq T_{\pi(\fp)} =T_{\pi(\fp)} / \pi(\fp) T_{\pi(\fp)}= \kappa(\pi(\fp))$.
\end{proof}

\section{The Field Extension $\sF_n \subseteq \sF_m$}

We denote by $\sF_{n}$ the field of rational functions $k(A)(p_1(Ax),\dots,p_n(Ax))$; this is a subfield of $\sE = k(A,x)$, this latter coinciding with $\sF_m$ by Proposition \ref{prp:Fm=E}. The main result of this section is:

\begin{proposition} \label{prp:F_n to F_m}
The field extension $\sF_n \subset \sF_m$ is algebraic of degree $[\sF_m : \sF_n] = n!$.
\end{proposition}

Towards proving Proposition \ref{prp:F_n to F_m}, we prove a fundamental fact:

\begin{lemma} \label{lem:regular sequence}
The polynomials $p_1(Ax),\dots,p_n(Ax)$ are a regular sequence of $R=k(A)[x]$.
\end{lemma}
\begin{proof}
For a homogeneous ideal $J$ in a polynomial ring $P$ over a field, it follows from Serre's theorem on Hilbert functions \cite[Theorem 4.4.3]{bruns1998cohen} that the Hilbert polynomial of $P/J$ agrees with the Hilbert function of $P/J$ at degrees greater than or equal to the Castelnuovo-Mumford (CM) regularity of $J$. It can be seen via the Koszul complex that the CM-regularity of an ideal $J$ generated by a regular sequence of $n$ elements of degrees $1,2,3,\dots,n$ is $\reg(J)=n(n-1)/2+1$.
			
Now let $J$ be the ideal of $R$ generated by $p_1(Ax),\dots,p_n(Ax)$. Then the sequence $p_1(Ax),\dots,p_n(Ax)$ is regular if and only if $R/J$ has zero Krull dimension, if and only if the Hilbert polynomial of $R/J$ is the zero polynomial, if and only if $[J]_{a}=[\fm]_{a}$; here $a=\reg(J) = n(n-1)/2+1$, $[J]_a$ denotes the degree-$a$ homogeneous part of $J$, and $\fm$ is the ideal of $R$ generated by $x_1,\dots,x_n$.
			
Let $t$ be the dimension of $[\fm]_a$ as a $k(A)$-vector space. Take generators for $[J]_a$ by multiplying every $p_i(Ax)$ by all monomials of degree $a-i$. Make a matrix $H$ whose columns contain the coefficients of the generators of $[J]_a$ on the basis of $[\fm]_a$ of all monomials of degree $a$. Then the $p_i(Ax)$'s are a regular sequence if and only if not all $t\times t$ minors of $H$ are zero. But if they were zero, they would also be zero upon substitution of $A$ by $A^*$, for any $A^* \in \bar{k}^{m \times n}$; here $\bar{k}$ denotes the algebraic closure of $k$. However, this would contradict the fact that the $p_i(A^* x)$'s are a regular sequence for a generic choice of $A^* \in \bar{k}^{m \times n}$, as per Lemma 4 in \cite{Tsakiris-TIT2020}; here we have used the fact that when the $p_i(A^* x)$'s are a regular sequence of $\bar{k}[x]$, the CM-regularity of the ideal they generate is still $a$, while the $\bar{k}$-vector space dimension of the degree-$a$ graded component of $\bar{k}[x]$ is still equal to $t$, so that the $p_i(A^* x)$'s are a regular sequence if and only if not all $t \times t$ minors of $H|_{A^*}$ are zero, where $H|_{A^*}$ is obtained from $H$ by replacing $A$ by $A^*$.
\end{proof}

Our next ingredient is fundamental as well and interesting in its own right.

\begin{lemma} \label{lem:n!}
$R=k(A)[x]$ is a free $k(A)[p_1(Ax),\dots,p_n(Ax)]$-module of rank $n!$.
\end{lemma}
\begin{proof}
We first prove that $R$ is flat over $R_n = k(A)[p_1(Ax),\dots,p_n(Ax)]$.
Quite generally, the inclusion of $R_n$ to $R$ is a homomorphism of positively graded rings, which takes the maximal homogeneous ideal of $R_n$ into the maximal homogeneous ideal $\fm$ of $R$. By part (3) of the Remark at page 178 in \cite{matsumura1989commutative}, $R$ is flat over $R_n$ if and only if $R_\fm$ is flat over $R_n$. Now, by Lemma \ref{lem:regular sequence} $p_1(Ax),\dots,p_n(Ax)$ is a regular sequence of $R$; as such it remains a regular sequence in $R_\fm$. Consequently, Theorem 1 in \cite{hartshorne1966property} gives that $R_\fm$ is flat over $R_n$; we conclude that $R$ is flat over $R_n$.

Next, we argue that $R$ is finitely generated over $R_n$. Let $J$ be the ideal of $R$ generated by $p_i(Ax), \, i \in [n]$, which is a regular sequence by Lemma \ref{lem:regular sequence}. It follows that $R/J$ is Artinian and of $k(A)$-vector space dimension $n!$; for the latter statement we used the known form for the Hilbert series of the quotient of $R$ by a regular sequence of homogeneous elements. Let $\cB$ be a $k(A)$-vector space basis of $R/J$. If $f \in R$ is a homogeneous polynomial of degree $d$, then we can write $f = \sum_{b \in \cB} c_b \, b + f_1$, where $c_b \in k(A)$ and $f_1$ is a homogeneous polynomial of $J$ of degree $d$. Writing $f_1 = \sum_{i \in [n]} g_i p_i(Ax)$, where $g_i$ is homogeneous of degree $d-i$, we have $f = \sum_{b \in \cB} c_b \, b + \sum_{i \in [n]} g_i p_i(Ax)$. Writing each $g_i$ as a $k(A)$-linear combination of the $b$'s plus a homogeneous polynomial of degree $d-i$ in $J$, we inductively see that $R$ is finitely generated over $R_n$ by $\cB$.

We now argue that $R$ is free over $R_n$ of finite rank. As $R$ is a finitely generated $R_n$-module and $R_n$ is a polynomial ring, $R$ is finitely presented over $R_n$; since $R$ is flat over $R_n$, it follows by the Corollary at page 53 in \cite{matsumura1989commutative} that $R$ is projective over $R_n$. Now, $R$ is a finitely generated graded module over the polynomial ring $R_n$, and so by Theorem 1.5.15(d) in \cite{bruns1998cohen} it is projective over $R_n$ if and only if it is free over $R_n$; we conclude that $R$ is a graded free $R_n$-module of finite rank.

Finally, we argue that the rank of $R$ as a free module over $R_n$ is $n!$. Let $\cU$ be homogeneous free generators of $R$ over $R_n$; it is clear that $\cU$ spans the $k(A)$-vector space $R/J$. Let us show that $\cU$ is linearly independent over $k(A)$ in $R/J$. Suppose $\sum_{u \in \cU} c_u \, u \in J$ for $c_u \in k(A)$. Since $J$ is a homogeneous ideal, we may assume that all $u$'s in the summation for which $c_u \neq 0$ have the same degree $d$. Then $\sum_{u \in \cU} c_u \, u + \sum_{i \in [n]} f_i p_i(Ax) = 0$ for some homogeneous $f_i \in R$ of degree $d-i$. Since $\cU$ generates $R$ over $R_n$, there are expressions $f_i = \sum_{u \in \cU} \rho_{i,u} \, u$, where $\rho_{i,u} \in R_n$. Substituting, we get a relation
\begin{equation*}
    \sum_{u \in \cU} \left(c_u  + \sum_{i \in [n]} \rho_{i,u} p_i(Ax) \right) u = 0.
\end{equation*}
\noindent Since the $u$'s are free over $R_n$ we must have $c_u  + \sum_{i \in [n]} \rho_{i,u} p_i(Ax) = 0$. The left-hand-side is a polynomial in $R_n$ whose constant term is $c_u$; it follows that $c_u = 0$.
\end{proof}

\begin{remark}
The proof of Lemma \ref{lem:n!} directly generalizes to any regular sequence $h_1,\dots,h_n$ of $R$ consisting of homogeneous elements of degrees $d_1,\dots,d_n$; this yields that $R$ is a free $k(A)[h_1,\dots,h_n]$-module of rank $d_1\cdots d_n$.
\end{remark}

We can now prove Proposition \ref{prp:F_n to F_m}.

\begin{proof}(Proposition \ref{prp:F_n to F_m})
For convenience we set $R_n = k(A)[p_1(Ax),\dots,p_n(Ax)]$. We have an inclusion of integral domains $R_n \subset R$, whose fields of fractions we denote respectively by $K(R_n)$ and $K(R)$. Now $K(R_n)$ is just a localization of $R_n$ and so it is flat over $R_n$. We thus have an inclusion of rings $K(R_n) \subset R \otimes_{R_n} K(R_n) \subset K(R)$. By Lemma \ref{lem:n!}, $R$ is a finitely generated $R_n$-module and so the ring extension $R_n \subset R$ is integral. It follows that the ring extension $K(R_n) \subset R \otimes_{R_n} K(R_n)$ is also integral; the Krull dimensions of both rings must be equal, whence $R \otimes_{R_n} K(R_n)$ is Artinian because $K(R_n)$ is a field. Quite generally, an Artinian ring contained in an integral domain must be a field; this gives that $R \otimes_{R_n} K(R_n)$ is a field. This field contains $R$ and, since it is contained in $K(R)$, it must be that $R \otimes_{R_n} K(R_n) = K(R)$. As $R$ is a free $R_n$-module of rank $n!$ by Lemma \ref{lem:n!}, it follows that $K(R)$ is a $K(R_n)$-vector space of dimension $n!$. Finally, recall that the fields $K(R_n)$ and $K(R)$ are just the fields $\sF_n$ and $\sF_m$, respectively.
\end{proof}

\section{The Field Extension $\sF_{n+1} \subseteq \sF_m$}

In this section we will prove that the field $\sF_{n+1} = k(A)(p_1(Ax),\dots,p_{n+1}(Ax))$ coincides with the field $\sF_m = \sE$; Theorem \ref{thm:fields} will follow as a special case for $k= \RR$.

We have fields $\sF_n \subseteq \sF_{n+1} \subseteq \sF_m$. By Proposition \ref{prp:F_n to F_m} the degree of the field extension $\sF_{n} \subset \sF_m$ is $n!$; hence to prove $\sF_{n+1} = \sF_m$ it suffices to prove that the degree of the field extension $\sF_n \subseteq \sF_{n+1}$ is $n!$. Note that $\sF_{n+1}$ is just the field $\sF_n(p_{n+1}(Ax))$ generated over $\sF_n$ by $p_{n+1}(Ax)$; thus it suffices to prove that the minimal polynomial $\mu_{n+1} \in \sF_n[t]$ of $p_{n+1}(Ax)$ over $\sF_n$ has degree $n!$. We will achieve this by using multidimensional resultants \cite{Macaulay_1916, Waerden_1950, Gelfand_Kapranov_Zelevinsky_1994,lang2012algebra}, which for the convenience of the reader we now briefly review following \cite{Macaulay_1916}.

For a positive integer $a$ and a field $\sK$ of characteristic zero, let $l_1,\dots,l_a$ be positive integers and denote by $\cM(t,l_i)$ the set of all monomials of degree $l_i$ in the variables $t=t_1,\dots,t_a$. For every $i \in [a]$ and every monomial $w \in \cM(t,l_i)$ we consider a variable $c(i,w)$ and define $f_i = \sum_{w \in \cM(t,l_i)} \, c(i,w) \, w$; this can be viewed as a homogeneous  polynomial of degree $l_i$ in the variables $t$ with coefficients in the polynomial ring $C=\sK[c(i,w): \, w \in \cM(t,l_i), \, i \in [a]]$. Set $l = l_1+\cdots+l_a-a+1$ and let $H$ be the matrix with entries in the ring $C$ defined as follows: with $w \in \cM(t,l-l_i)$ and $i \in [m]$ fixed, we consider the column-vector that gives the coefficients in $C$ of $w f_i$ with respect to $\cM(t,l)$; then $H$ has as its columns all such vectors as $i$ and $w$ range in $[a]$ and $\cM(t,l-l_i)$ respectively. Macaulay defined the resultant $\cR(f_1,\dots,f_a)$ of $f_1,\dots,f_a$ as the greatest common divisor of all maximal minors of $H$; he proved that it has the following properties that we shall need:

\begin{proposition}[\cite{Macaulay_1916}, \S6-\S10, Chapter I] \label{prp:Macaulay}
Set $L = l_1\cdots l_a$ and $L_i  = L / l_i$. Then
\begin{enumerate}
\item For $i \in [a]$ the degree of $c(i,t_i^{l_i})$ in $\cR(f_1,\dots,f_a)$ is $L_i$ and the coefficient of $c(a,t_a^{l_a})^{L_a}$ is $\cR(\bar{f}_1,\dots,\bar{f}_{a-1})^{l_a}$, where $\bar{f}_i = f_i|_{t_a=0}$.
\item For every $i \in [a]$, we have that $\cR(f_1,\dots,f_a)$ is homogeneous in the variables $(c(i,w))_{w\in\cM(t,l_i)}$ of degree $L_i$.
\item Let $f_1^*,\dots,f_a^* \in \sK[t]=\sK[t_1,\dots,t_a]$ be any specialization of $f_1,\dots,f_a$, obtained by replacing each $c(i,w)$ by an element of $\sK$, and $\cR(f_1^*,\dots,f_a^*)$ the corresponding specialization of $\cR(f_1,\dots,f_a)$. Then $\cR(f_1^*,\dots,f_a^*) = 0$ if and only if $f_1^*,\dots,f_a^*$ have a common root in $\bar{\sK}^a$ besides zero; here $\bar{\sK}$ is the algebraic closure of $\sK$.
\end{enumerate}
\end{proposition}

\begin{remark} \label{rem:regular-resultant}
It is a basic observation that a set of $n$ homogeneous polynomials in a polynomial ring of dimension $n$ over $\sK$ are a regular sequence if and only if the ideal they generate is primary to the maximal homogeneous ideal, which is equivalent to the polynomials admitting no common root in $\bar{\sK}^n$ other than zero.  Indeed, the matrix $H$ that appeared in the proof of Lemma \ref{lem:regular sequence} is precisely Macaulay's matrix associated to the resultant of $p_1(Ax),\dots,p_n(Ax)$.
\end{remark}

We now return to our objective of showing that the minimal polynomial $\mu_{n+1}$ of $p_{n+1}(Ax)$ over $\sF_n$ has degree $n!$. We apply the formulation above with $a=n+1$, letting $t$ be the column vector with entries $t_1,\dots,t_n$, and introducing new variables $r_1,\dots,r_{n+1}$. For $i \in [n+1]$ we define $f_i^* = p_i(At)-r_i t_{n+1}^i$; this is a homogeneous polynomial of degree $i$ in the variables $ t_1,\dots,t_{n+1}$ with coefficients in the polynomial ring $k(A)[r]:=k(A)[r_1,\dots,r_{n+1}]$. We denote by $\cR(f_1^*,\dots,f_{n+1}^*) \in k(A)[r]$ the specialization of the resultant $\cR(f_1,\dots,f_m)$ of Proposition \ref{prp:Macaulay} with the variable $c(i,w)$ replaced by the corresponding coefficient of $w$ in $f_i^*$. Similarly, applying the formulation above with $a=n$, we let $\cR(p_1(At),\dots,p_n(At))$ be the specialization of $\cR(f_1,\dots,f_n)$ with the variable $c(i,w)$ replaced by the corresponding coefficient of $w$ in $p_i(At)$. It will be convenient to explicitly indicate the dependence of $\cR(f_1^*,\dots,f_{n+1}^*) \in k(A)[r]$ on the variables $A$ and $r=r_1,\dots,r_{n+1}$; for this we shall write $\rho(A,r_1,\dots,r_{n+1}):= \cR(f_1^*,\dots,f_{n+1}^*).$ We next proceed with a series of key observations.

\begin{lemma} \label{lem:coefficient}
$\rho(A,r_1,\dots,r_{n+1}) \in k(A)[r]$ is a non-zero polynomial of degree $n!$ in the variable $r_{n+1}$; moreover, the coefficient of $r_{n+1}^{n!}$ is $\cR(p_1(Ax),\dots,p_n(Ax))^{n+1}$.
\end{lemma}
\begin{proof}
By Lemma \ref{lem:regular sequence} $p_1(Ax),\dots,p_n(Ax)$ is a regular sequence of $R = k(A)[x]$. Part (3) of Proposition \ref{prp:Macaulay} and Remark \ref{rem:regular-resultant} give that $\cR(p_1(Ax),\dots,p_n(Ax))$ is a non-zero element of $k(A)$. The statement now follows from part (1) of Proposition \ref{prp:Macaulay}, because the resultant commutes with specialization.
\end{proof}

\begin{lemma} \label{lem:rho-vanishing}
We have $\rho \left(A,p_1(Ax),\dots,p_{n+1}(Ax)\right) = 0$.
\end{lemma}
\begin{proof}
After substituting $r_i$ with $p_i(Ax)$ in the polynomials $f_i^*$ for every $i \in [n+1]$, it is evident that the point $(x_1,\dots,x_n,1) \in \mathscr{E}^{n+1}$ is a common root; thus the resultant of these polynomials vanishes by part (4) of Proposition \ref{prp:Macaulay}.
\end{proof}

\begin{lemma} \label{lem:rho-non-vanishing}
We have that $\rho\left(A,p_1(Ax),\dots,p_{n}(Ax),r_{n+1} \right) \neq 0$.
\end{lemma}
\begin{proof}
This follows immediately from Lemma \ref{lem:coefficient}.
\end{proof}

\begin{remark} \label{rem:irreducible}
As by Lemma \ref{lem:regular sequence} the polynomials $p_1(Ax),\dots,p_n(Ax)$ are a regular sequence of $k(A)[x]$, they are algebraically independent over $k(A)$. Hence the polynomials $p_1(Ax),\dots,p_n(Ax),r_{n+1}$ are algebraically independent over $k(A)$. 
Therefore, $\rho\left(A,p_1(Ax),\dots,p_{n}(Ax),r_{n+1} \right)$ is irreducible in $k(A)[p_1(Ax),\dots,p_{n}(Ax),r_{n+1}]$ if and only if $\rho\left(A,r_1,\dots,r_{n+1} \right)$ is irreducible in $k(A)[r_1,\dots,r_{n+1}]$.
\end{remark}

We are going to prove that $\rho\left(A,r_1,\dots,r_{n+1} \right)$ is irreducible in $k[A,r_1,\dots,r_{n+1}]$. We do this by first proving the special case $m = n+1$, where all the difficulty concentrates.

\begin{lemma} \label{lem:irreducible(n+1)}
Suppose $m=n+1$, then $\rho\left(A,r_1,\dots,r_{n+1} \right)$ is irreducible as a polynomial in $k[A,r_1,\dots,r_{n+1}]$.
\end{lemma}
\begin{proof}
When $m=n+1$ we have $\sF_m = \sF_{n+1}$ and the degree of the field extension $\sF_n \subseteq \sF_{n+1} $ is $n!$ by Proposition \ref{prp:F_n to F_m}. Since $\sF_{n+1} = \sF_n(p_{n+1}(Ax))$, we have that the minimal polynomial $\mu_{n+1}$ of $p_{n+1}(Ax)$ over $\sF_n$ is of degree $n!$. On the other hand, by Lemma \ref{lem:rho-non-vanishing} the polynomial $\rho\left(A,p_1(Ax),\dots,p_{n}(Ax),r_{n+1} \right) \in \sF_n[r_{n+1}]$ is a polynomial of degree $n!$, which by Lemma \ref{lem:rho-vanishing} has $p_{n+1}(Ax)$ as its root. Consequently, $\rho\left(A,p_1(Ax),\dots,p_{n}(Ax),r_{n+1} \right) \in \sF_n[r_{n+1}]$ is the minimal polynomial $\mu_{n+1}$ of $p_{n+1}(Ax)$ over $\sF_n$ up to multiplication by an element of $\sF_n$.
In view of Remark \ref{rem:irreducible}, we have that $\rho\left(A,r_1,\dots,r_{n+1} \right) \in k(A)(r_1,\dots,r_n)[r_{n+1}]$ is irreducible, where $k(A)(r_1,\dots,r_n)$ denotes the field of fractions of the polynomial ring $k(A)[r_1,\dots,r_n]$. By Gauss's lemma on irreducible polynomials, it suffices to prove that $\rho\left(A,r_1,\dots,r_{n+1} \right)$ is primitive in $k[A,r_1,\dots,r_n][r_{n+1}]$.

Any common factor of the coefficients of $\rho\left(A,r_1,\dots,r_{n+1} \right)$ must divide the coefficient of $r_{n+1}^{n!}$, which by Lemma \ref{lem:coefficient} is $\cR(p_1(Ax),\dots,p_n(Ax))^{n+1} \in k[A]$. Hence, it suffices to show that no non-constant polynomial $p(A) \in k[A]$ divides $\rho\left(A,r_1,\dots,r_{n+1} \right)$. Suppose otherwise, that is $p(A) \in k[A]$ is a non-constant factor of $\cR(A,r_1,\dots,r_{n+1})$ and let $A^* \in \bar{k}^{m \times n}$ be a root of $p(A)$. It follows that for any choice $(r_1^*,\dots,r_{n+1}^*) \in \bar{k}^{n+1}$ we have $\cR(A^*,r_1^*,\dots,r_{n+1}^*) = 0$, so that by
part (3) of Proposition \ref{prp:Macaulay} the polynomials $f_i^*(x) = p_i(A^*x)-r_i^* x_{n+1}^i \in \bar{k}[x_1,\dots,x_{n+1}], \, i \in [n+1]$ have a common root $0 \neq (\xi_1,\dots,\xi_{n+1}) \in \bar{k}^{n+1}$. Let us distinguish between the case $\rank(A^*) = n$ and $\rank(A^*) < n$.

Suppose that $\rank(A^*) = n$. If $\xi_{n+1}=0$, letting $\xi$ be the column vector $(\xi_1,\dots,\xi_n) \in \bar{k}^n$, we have that $A^*\xi \in \bar{k}^{m}$ is a common root of the polynomials $p_1(z),..,p_{m}(z) \in \bar{k}[z] = \bar{k}[z_1,\dots,z_{m}]$ (recall $m = n+1$). Since the $p_i$'s are a regular sequence in $\bar{k}[z]$, we must have that $A^* \xi = 0$. By hypothesis $\rank(A^*) = n$ so that $\xi = 0$; however, this contradicts our assumption that not all $\xi_1,\dots,\xi_{n+1}$ are zero. Hence, it must be that $\xi_{n+1} \neq 0$, and since the $f_i^*$'s are homogeneous, we may assume $\xi_{n+1} = 1$; in turn, this gives $r_i^* = p_i(A^*\xi)$ for every $i \in [n+1]$. As the $r_i^*$'s were chosen arbitrarily, the image of the polynomial map $\bar{k}^n \rightarrow \bar{k}^{n+1}$ that takes $\beta^* \in \bar{k}^n$ to $\left(p_1(A^*\beta^*),\dots,p_{n+1}(A^*\beta^*)\right) \in \bar{k}^{n+1}$ must be the entire $\bar{k}^{n+1}$. This implies that the affine coordinate ring of the closure of this map, which is $\bar{k}[p_1(A^*x),\dots,p_{n+1}(A^*x)]$, must have Krull dimension $n+1$. But this is impossible because this is a subring of $\bar{k}[x]=\bar{k}[x_1,\dots,x_n]$.

We have concluded that $\rank(A^*)<n$ for any root $A^*$ of $p(A)$. In other words, the hypersurface of $\bar{k}^{m \times n}$ defined by the polynomial $p(A)$ must lie in the determinantal variety defined by the ideal $I_{n}(A)$ of maximal minors of $A$. But this is impossible, because the dimension of the former is $(n+1)n-1=n^2+n-1$, while the dimension of the latter is well-known to be $(n-1)(n+2) = n^2+n-2$ \cite{bruns2006determinantal, Aldo-book}.
\end{proof}

We now treat the general case.

\begin{lemma} \label{lem:irreducible(m)}
$\rho\left(A,r_1,\dots,r_{n+1} \right)$ is irreducible in $k[A,r_1,\dots,r_{n+1}]$.
\end{lemma}
\begin{proof}
The case $m=n+1$ has been proved in Lemma \ref{lem:irreducible(n+1)}, hence we assume $m>n+1$. Let us denote by $\bar{A}$ the matrix obtained from $A$ by replacing all $a_{ij}$'s for which $i>n+1$ with zero; this induces a $k$-algebra homomorphism $\vartheta: k[A,r] \rightarrow k[\bar{A},r]$ which takes $A$ to $\bar{A}$. As the resultant commutes with specialization, $\vartheta\left(\rho\left(A,r_1,\dots,r_{n+1} \right) \right) = \rho\left(\bar{A},r_1,\dots,r_{n+1} \right)$. Now suppose $\rho\left(A,r_1,\dots,r_{n+1} \right) = g h$ with $g,h$ non-constant polynomials in $k[A,r]$; then $\rho\left(\bar{A},r_1,\dots,r_{n+1} \right) = \vartheta(g) \vartheta(h)$ with $\vartheta(g), \, \vartheta(h) \in k[\bar{A},r_1,\dots,r_{n+1}]$. By Lemma \ref{lem:irreducible(n+1)} $\rho\left(\bar{A},r_1,\dots,r_{n+1} \right)$ is irreducible in $k[\bar{A},r_1,\dots,r_{n+1}]$, so we may assume that $\vartheta(g) = 1$. It follows that the degree of $r_{n+1}$ in $\vartheta(h)$ is $n!$, whence the degree of $r_{n+1}$ in $h$ is $n!$ as well. In turn, this gives that $g$ divides the coefficient of $r_{n+1}^{n!}$ in $\rho\left(A,r_1,\dots,r_{n+1} \right)$, which by Lemma \ref{lem:coefficient} is $\cR(p_1(Ax),\dots,p_n(Ax))^{n+1} \in k[A]$. By part (2) of Proposition \ref{prp:Macaulay}, for every $i \in [n]$, we have that $\cR(p_1(Ax),\dots,p_n(Ax))$ is a homogeneous polynomial in the coefficients of $p_i(Ax)$ of degree $n! / i$; those coefficients are homogeneous polynomials themselves in the variables $A$ of degree $i$. We thus conclude that $\cR(p_1(Ax),\dots,p_n(Ax))$ is a homogeneous polynomial in the variables $A$ of degree $n n!$. Since $g$ divides $\cR(p_1(Ax),\dots,p_n(Ax))^{n+1}$ it must also be homogeneous, and the fact that $\vartheta(g) = 1$ shows that $g$ is indeed a constant polynomial.
\end{proof}

We have arrived at the following crucial fact:

\begin{lemma} \label{lem:minimal polynomial}
The minimal polynomial $\mu_{n+1}$ of $p_{n+1}(Ax)$ over $\sF_n$ is up to multiplication by an element of $\sF_n$ equal to $\rho(A,p_1(Ax),\dots,p_{n}(Ax),r_{n+1})$, and thus has degree $n!$.
\end{lemma}
\begin{proof}
By Lemma \ref{lem:irreducible(m)} and Remark \ref{rem:irreducible} $\rho(A,p_1(Ax),\dots,p_{n}(Ax),r_{n+1})$ is irreducible as a polynomial in $k[A,p_1(Ax),\dots,p_n(Ax)][r_{n+1}]$. By Gauss's lemma on irreducible polynomials over a unique factorization domain, $\rho(A,p_1(Ax),\dots,p_{n}(Ax),r_{n+1})$ is also irreducible in $k(A)(p_1(Ax),\dots,p_n(Ax))[r_{n+1}]=\sF_n[r_{n+1}]$. By Lemma \ref{lem:coefficient} it has degree $n!$ in $r_{n+1}$ and by Lemma \ref{lem:rho-vanishing} it has $p_{n+1}(Ax)$ as its root.
\end{proof}

We can now state and prove the main technical theorem of this paper:

\begin{theorem} \label{thm:F(n+1) = F(m)}
We have an equality of fields $\sF_{n+1} = \sF_m$.
\end{theorem}
\begin{proof}
We have $[\sF_m : \sF_n] = n!$ by Proposition \ref{prp:F_n to F_m} and $[\sF_{n+1} : \sF_n] = n!$ by Lemma  \ref{lem:minimal polynomial}; since $\sF_n \subseteq \sF_{n+1} \subseteq \sF_m$, we must have $\sF_{n+1} = \sF_m$.
\end{proof}

As a corollary to Theorem \ref{thm:F(n+1) = F(m)}, we can prove Theorem \ref{thm:main}.

\begin{proof}(Theorem \ref{thm:main})
We have $\sF_{n+1} = \sF_m$ by Theorem \ref{thm:F(n+1) = F(m)} and $\sF_m = \sE$ by Proposition \ref{prp:Fm=E}; that is $\sF_{n+1} = \sE$. Concretely,
$k(A)(p_1(Ax),\dots,p_{n+1}(Ax)) = k(A)(x).$ It immediately follows from this equality that each $x_i$ is a rational function over $k$ in $A,p_1(Ax),\dots,p_{n+1}(Ax)$. In particular, for every $i \in [n]$, there exist polynomials \begin{equation*}
    f_i\big(A,p_1(Ax),\dots,p_{n+1}(Ax)\big),g_i\big(A,p_1(Ax),\dots,p_{n+1}(Ax)\big) 
\end{equation*}
\noindent in $ k[A][p_1(Ax),\dots,p_{n+1}(Ax)]$ such that $x_i = f_i / g_i$; in fact, since this holds for every field $k$ of characteristic zero, one sees that $f_i, \, g_i \in \ZZ[A][p_1(Ax),\dots,p_n(Ax)].$

Now let $A^* \in \bar{k}^{m \times n}$ and $x^* \in \bar{k}^n$ be generic in the sense that none of the $g_i$'s evaluates to zero upon substitution of $A$ and $x$ with $A^*$ and $x^*$, respectively. Suppose that $\xi \in \bar{k}^n$ is a common root of the polynomials $q_i(x) = p_i(A^*x) - p_i(A^*x^*), \, i \in [n+1]$; that is $p_i(A^* \xi) = p_i(A^* x^*)$ for every $i \in [n+1]$. As a consequence,
\begin{equation*}
    0 \neq g_i\big(A^*,p_1(A^*x^*),\dots,p_{n+1}(A^*x^*)\big) = g_i\big(A^*,p_1(A^*\xi),\dots,p_{n+1}(A^*\xi)\big)
\end{equation*}
\noindent $\forall i \in [n]$, and thus the equality of rational functions $x_i = f_i / g_i$ gives an equality in $\bar{k}$

\begin{equation*}
    \xi_i =  \frac{f_i\big(A^*,p_1(A^*\xi),\dots,p_{n+1}(A^*\xi)\big)}{g_i\big(A^*,p_1(A^*\xi),\dots,p_{n+1}(A^*\xi)\big)} = \frac{f_i\big(A^*,p_1(A^*x^*),\dots,p_{n+1}(A^*x^*)\big)}{g_i\big(A^*,p_1(A^*x^*),\dots,p_{n+1}(A^*x^*)\big)}  = x_i^*
\end{equation*}

\noindent for every $i \in [n]$.
\end{proof}

\bibliographystyle{siamplain}
	\bibliography{DCHS,LiangBook,FTVUS23,MRP,LiangArticle, Liangzu,UPCAManolis}
\end{document}